\documentclass[12pt,a4paper,reqno]{amsart}

\usepackage[lmargin=1.4 in, rmargin= 1.4in, tmargin=1.7in]{geometry}

\title{The Category of Necklaces is Reedy Monoidal}

\author{Violeta Borges Marques}
\address[Violeta Borges Marques]{Universiteit Antwerpen, Departement Wiskunde, Middelheimcampus, Middelheimlaan 1, 2020 Antwerp, Belgium}
\email{Violeta.BorgesMarques@uantwerpen.be}

\author{Arne Mertens}
\address[Arne Mertens]{Universiteit Antwerpen, Departement Wiskunde, Middelheimcampus, Middelheimlaan 1, 2020 Antwerp, Belgium}
\email{arne.mertens@uantwerpen.be}

\thanks{
This project has received funding from the European Research Council (ERC) under the European Union’s Horizon 2020 research and innovation programme (grant agreement No. 817762).
}

\makeatletter
\@namedef{subjclassname@2020}{
  \textup{2022} Mathematics Subject Classification}
\makeatother



\usepackage{graphicx}
\usepackage[english]{babel}
\usepackage{amsmath}
\usepackage{amssymb}
\usepackage{amsthm}
\usepackage[all,cmtip]{xy}
\usepackage{cancel}
\usepackage[alpine]{ifsym}
\usepackage{enumerate}
\usepackage{stmaryrd}
\usepackage{tikz}
\usepackage{tikz-cd}
\usetikzlibrary{matrix}
\usepackage[toc,page]{appendix}
\usepackage{hyperref}
\usepackage{mathtools}
\usepackage{float}
\usepackage{hhline}
\usepackage{comment}
\usepackage[style=alphabetic, maxnames = 50, maxalphanames=50, urldate=long, firstinits=true]{biblatex}
\DeclareMathOperator{\cvee}{\vee\cdots\vee}

\DeclareMathOperator{\Ob}{Ob}

\DeclareMathOperator{\id}{id}

\DeclareMathOperator{\Fun}{Fun}

\DeclareMathOperator*{\colim}{colim}
\DeclareMathOperator{\im}{im}

\DeclareMathOperator{\Ch}{Ch}

\DeclareMathOperator{\SSet}{SSet}

\DeclareMathOperator{\Cat}{Cat}

\DeclareMathOperator{\sk}{sk}
\DeclareMathOperator{\cosk}{cosk}

\newcommand{\fint}{\mathbf{\Delta}_{f}} 
\newcommand{\simp}{\mathbf{\Delta}} 
\newcommand{\nec}{\mathcal{N}ec} 
\newcommand{\ts}{S_{\otimes}} 

\newcommand{\N}{\mathbb{N}}

\newcommand{\rrr}{\mathcal{R}} 
\newcommand{\sss}{\mathcal{S}} 

\newcommand{\op}{\mathrm{op}}

\DeclareMathOperator{\dayconv}{\otimes_{\Day}}

\newcommand{\ot}{\leftarrow}
\DeclareMathOperator{\ac}{ac}
\DeclareMathOperator{\ine}{in}
\DeclareMathOperator{\Reedy}{Reedy}

\DeclareMathOperator{\Day}{Day}

\newtheorem{Thm}{Theorem}[section]
\newtheorem*{Thm*}{Theorem}
\newtheorem{Lem}[Thm]{Lemma}
\newtheorem{Prop}[Thm]{Proposition}
\newtheorem{Cor}[Thm]{Corollary}

\theoremstyle{definition}
\newtheorem{Def}[Thm]{Definition}

\newtheorem{Exs}[Thm]{Examples}

\theoremstyle{remark}
\newtheorem{Rem}[Thm]{Remark}

\begin{document}

\begin{abstract}
In the first part of this note we further the study of the interactions between Reedy and monoidal structures on a small category, building upon the work present in \cite{Barwick}. We define a \emph{Reedy monoidal category} as a Reedy category $\rrr$ which is monoidal such that for all symmetric monoidal model categories $\textbf{A}$, the category $\Fun\left(\rrr^{\op}, \textbf{A}\right)_{\Reedy}$ is monoidal model when equipped with the Day convolution. In the second part, we study the category $\nec$ of necklaces, as defined in \cite{baues1980geometry}\cite{dugger2011rigidification}. Making use of a combinatorial description present in \cite{grady2023extended}\cite{lowen2023enriched}, we streamline some proofs from the literature, and finally show that $\nec$ is simple Reedy monoidal.
\end{abstract}

\maketitle

\tableofcontents

\section{Introduction}

\par Reedy categories were first defined by Kan in unpublished notes and provide an abstract setting to generalize \cite[Lemma 1.2]{reedy1974}, where simplicial objects are inductively defined through the factorization of the canonical maps $\sk_n X \to \cosk_n X$. This inductive construction allows one to define a model structure, the Reedy model structure, on functor categories $\Fun\left(\rrr^{\op}, \textbf{A}\right)$. As opposed to the projective and injective model structure \cite{lurie2009higher}, the existence result for the Reedy model structure relies on imposing strong conditions on the small category $\rrr$ rather than on the model category $\textbf{A}$. The Reedy model structure also has the advantage that both (trivial) fibrations and cofibrations are explicitly prescribed.
\par Necklaces first appeared in \cite{baues1980geometry} (under the name ``cellular strings'') in the study of the bar and cobar construction and their relation to loop spaces. The terminology ``necklace'' was introduced and popularized by Dugger and Spivak, in \cite{dugger2011rigidification}\cite{dugger2011mapping} to provide several homotopically equivalent models of the mapping spaces of a quasi-category. Since then, necklaces have found further applications in work by Rivera and others, in particular in the construction of models of path spaces \cite{CategoricalPath}\cite{CombinatorialPath}; an adaptation of necklaces, called closed necklaces, was used to construct models of free loop spaces in \cite{AlgebraicFree} and \cite{CombinatorialFree}. In \cite{Einfty} a strengthening of Baues' original results was presented. Several generalizations that use necklaces followed, for example, for dendroidal $\infty$-operads in \cite{DendroidalNecklace}, for cubical quasi-categories in \cite{curien2022rigidification}, for cartesian enriched quasi-categories in \cite{gindi2019rigidification} and not necessarily cartesian enriched in \cite{lowen2023enriched}\cite{mertens2023nerves}. The category of necklaces has also been employed to construct a Segalification functor, providing a Segal space generated by a simplicial space, in \cite{barkan2023segalification}, to treat concurrency problems in \cite{concurrency} and to study (fully) extended functorial field theories (FFT) in \cite{grady2023extended}.
\par In the present paper, we consider the model category $\Fun(\rrr^{op},\textbf{A})_{\Reedy}$ for a Reedy category $\rrr$ and a symmetric monoidal model category $\textbf{A}$. This category is naturally endowed with the pointwise monoidal structure and the question of whether this makes the Reedy model structure monoidal model is treated in \cite{ghazel2019}. On the other hand, if $\rrr$ is itself monoidal, we can also consider the Day convolution product \cite{day1970closed} on $\Fun(\rrr^{op},\textbf{A})$. We define $\rrr$ to be \emph{Reedy monoidal} (Definition \ref{Def:ReedyMonoidal}) if $\Fun(\rrr^{op},\textbf{A})_{\Reedy}$ is monoidal model with respect to the Day convolution for every symmetric monoidal model category $\textbf{A}$. Building on general model categorical results by Barwick \cite{Barwick}, we also provide simple combinatorial conditions to ensure that a category that is both Reedy and monoidal is indeed Reedy monoidal (Theorem \ref{Thm:ReedyMonoidal}).
\par Subsequently, we construct a Reedy structure on the category $\nec$ of necklaces (Theorem \ref{Thm:NecIsReedy}). The main ingredients of this structure are (at least implicitly) present in \cite{CombinatorialPath}\cite{rivera2018cubical}, but to the best of our knowledge the full Reedy (monoidal) structure has not been presented explicitly in the literature yet. For this, we make use of a combinatorial description of $\nec$ put forward in \cite{grady2023extended}\cite{lowen2023enriched} which also makes some existing results easier to prove. Finally, we show that $\nec$ is a simple Reedy monoidal category (Theorem \ref{Thm:NecIsReedyMonoidal}) so that $\Fun(\nec^{op},\textbf{A})_{\Reedy}$ is always monoidal model when equipped with the Day convolution.

\subsection{Motivation}
In \cite{simpson2012homotopy}, Simpson developed a theory of Segal categories enriched in a cartesian model category. Given an appropriate cartesian model category $\mathcal{M}$, a model structure on the category of unital $\mathcal{M}$-precategories is established, where the weak equivalences are the global weak equivalences and the fibrant objects satisfy the Segal condition (Theorems IV.19.2.1 and IV.19.4.1). The results of the present paper are motivated by an ongoing project to construct a model for Segal categories enriched in non-cartesian monoidal model categories $\mathbf{A}$ as well. A main example of interest is the projective model structure on chain complexes $\mathbf{A} = \Ch(k)$ over a field $k$. The unital $\mathbf{A}$-precategories in this case are no longer given by simplicial objects $S\mathbf{A}$, but by \emph{templicial objects} $\ts\mathbf{A}$. These are certain strictly unital colax monoidal functors which were introduced in \cite{lowen2023enriched} as replacements for simplicial objects in the non-cartesian context, and inspired by earlier work of Leinster \cite{leinster2000homotopy} and Bacard \cite{bacard2010Segal}\cite{bacard2013colax}. In the non-cartesian setting we propose then the following definition.

\begin{Def}
Let $\textbf{A}$ be a monoidal model category. A Segal category is a templicial $\textbf{A}$-object $(X,S)\in \ts\textbf{A}$ such that the comultiplication maps
    $\mu_{i,j}: X_{i+j}\xrightarrow[]{\sim} X_i \otimes_S X_j$
\noindent are weak equivalences for all $i,j\geq 0$.
\end{Def}

\par A first step in this direction is the definition of a Reedy model structure on $\ts \textbf{A}$ as follows. It is shown in \cite{lowen2023enriched} that there is an adjunction with fully faithful left adjoint
\[\begin{tikzcd}
	\ts\mathbf{A} & {\Fun(\nec^{op},\mathbf{A})\text{-}\Cat}
	\arrow[""{name=0, anchor=center, inner sep=0}, "{(-)^{nec}}", hook, shift left=2, from=1-1, to=1-2]
	\arrow[""{name=1, anchor=center, inner sep=0}, "{(-)^{temp}}", shift left=2, from=1-2, to=1-1]
	\arrow["\dashv"{anchor=center, rotate=-90}, draw=none, from=0, to=1]
\end{tikzcd}\]
In future work, we will show that under suitable yet moderate conditions, the monoidal model structure $\Fun\left(\nec^{\op}, \textbf{A}\right)_{\Reedy}$ induces one on  $\Fun\left(\nec^{\op}, \textbf{A}\right)$-$\Cat$ and $\Fun\left(\nec^{\op}, \textbf{A}\right)$-$\Cat_S$, the category of necklicial categories with fixed object set $S$. In case $\textbf{A}$ is cartesian or $\textbf{A}=\Ch k$, the former can be further transferred to $\ts \textbf{A}_S$. Moreover, in case $\textbf{A}$ is cartesian, we recover the classical Reedy structure under the equivalence $\ts \textbf{A}_S\cong \Fun\left(\Delta_S^{\op} /S, \textbf{A}\right)$ \cite[Proposition III.12.3.1]{simpson2012homotopy}.

\vspace{0,3cm}
\noindent \emph{Acknowledgement.}
The authors would like to thank Wendy Lowen for her helpful discussions and support during the writing of this paper. They are also grateful to Clemens Berger for drawing attention to the reference \cite{baues1980geometry}, and to Dmitry Kaledin for pointing out the work of Manuel Rivera.

\section{Reedy monoidal categories}
 \par We start this section by recalling some concepts of the theory of Reedy categories following \cite[\S 3]{Barwick}. While loc. cit. introduces these concepts with a model category theoretical approach, we opt to adopt the equivalent combinatorial characterizations as our primary definitions. For a treatment of the model category theory aspects, we defer to the latter part of this section.  

 In \S 2.1, we show that any morphism of Reedy categories that restricts to a discrete fibration between direct subcategories, is a right fibration (Proposition \ref{Prop:DirectDivImpliesRightFib}). 
 As a consequence, the monoidal product of any direct divisible Reedy category in the sense of \cite{bacard2013colax} is a right fibration. We conclude in \S 2.2 by introducing Reedy monoidal categories (Definition \ref{Def:ReedyMonoidal}), along with sufficient conditions for checking this property (Theorem \ref{Thm:ReedyMonoidal}).

\subsection{Reedy categories and right fibrations}

\par For standard treatments of the theory of Reedy categories, we refer to \cite{hovey1999model} \cite{hirschhorn2003model}. Let us fix a small category $\mathcal{R}$.

\begin{Def}[\cite{hovey1999model}, Definition 5.2.1]
\label{Def:ReedyCat}
A \emph{Reedy structure} on $\rrr$ is a pair of wide \emph{inverse} and \emph{direct} subcategories $(\rrr^\ot, \rrr^\to)$ and a \emph{degree function} $\deg: \Ob(\rrr)\rightarrow \lambda$ with $\lambda$ an ordinal, such that the following conditions are satisfied:
\begin{enumerate}
    \item Every morphism $f$ in $\rrr$ factors uniquely as $f = f^{\to}\circ f^{\ot}$ with $f^{\to}\in\rrr^{\to}$ and $f^{\ot}\in\rrr^{\ot}$.
    \item Every non-identity morphism in $\rrr^{\ot}$ lowers the degree and every non-identity morphism in $\rrr^{\to}$ raises the degree.
\end{enumerate}
\end{Def}

\begin{Exs}\label{Ex:ReedyCat}
\begin{enumerate}
    \item The simplex category $\simp$ is Reedy with the degree function $d: \Ob(\simp)\rightarrow \N: [n]\mapsto n$, and $\simp^{\ot}$ and $\simp^{\to}$ containing the surjective and injective order morphisms respectively.
    \item If $(\mathcal{R}, \rrr^\ot, \rrr^\to)$ and $(\mathcal{S}, \sss^{\ot}, \sss^{\to})$ are Reedy categories, then so are $(\rrr^{\op}, (\rrr^{\to})^{\op}, (\rrr^{\ot})^{\op})$ and $(\rrr \times \sss, \rrr^{\ot}\times \sss^{\ot}, \rrr^{\to}\times \sss^{\to})$.
\end{enumerate}
\end{Exs}

\begin{Def}[\cite{hovey1999model}, Def 5.1.2]
\label{Def:LatchingMatchingCats} Let $\mathcal{R}$ be a Reedy category and $\alpha\in \mathcal{R}$. The \emph{latching category} at $\alpha$ is
\[
\partial\left(\mathcal{R}^\to /\alpha\right)=\{f\in\mathcal{R}^\to /\alpha \mid f\neq \id\}
\]
\noindent and the \emph{matching category} at $\alpha$ is
\[
\partial\left(\alpha / \mathcal{R}^\ot \right)=\{f\in\alpha / \mathcal{R}^\ot  \mid f\neq \id\}.
\]
\end{Def}

\begin{Rem}
\label{Rem:LatchingMatchingOp}
Note that there is a canonical isomorphism between the latching (resp. matching) category of $\rrr$ at $\alpha$ and the matching (resp. latching) category of $\rrr^{\op}$ at $\alpha$.
\end{Rem}

\par Now let us turn to functors between Reedy categories.

\begin{Def}[\cite{Barwick}, Definition 3.16.1]
\label{Def:MorphismOfReedy}
Let $\mathcal{R}$ and $\mathcal{S}$ be two Reedy categories. A functor $\mathcal{R}\xrightarrow{F}\mathcal{S}$ is a \emph{morphism of Reedy categories} if $F(R^\to)\subseteq S^\to$ and $F(R^\ot)\subseteq S^\ot$.
\end{Def}

\begin{Def}[\cite{Barwick}, Theorem 3.22]
\label{Def:RightFibIffLatchingEmptyOrConnectec}
A morphism of Reedy categories $\mathcal{R}\xrightarrow{F}\mathcal{S}$ is called
\begin{itemize}
    \item a \emph{right fibration} if for every $\alpha\in\mathcal{R}$ and every $\beta \xrightarrow{f}F(\alpha) \in \beta / F$ the category $\partial\left(\left(\beta/F\right)^\to/f\right)$ is empty or connected. The category $\rrr$ is \emph{right fibrant} if the functor $\rrr\to\{*\}$ is a right fibration.
    \item a  \emph{left fibration} if for every $\alpha\in \rrr$ and every $F(\alpha)\xrightarrow{f}\beta\in F/\beta$ the category $\partial\left(f/\left(F/\beta\right)^{\ot}\right)$ is empty or connected. The category $\rrr$ is \emph{left fibrant} if the functor $\rrr\to\{*\}$ is a left fibration.
\end{itemize}
\end{Def}
\begin{Rem}
\label{Rem:CommaReedy}
In the previous definition we make implicit use of the Reedy structure on $\beta / F$ and $F / \beta$ given in \cite[Lemma 3.10]{Barwick}.
\end{Rem}

\begin{Rem}
\label{Rem:TerminalLeftFib}
Observe that $\rrr$ is left fibrant if and only if for any $\alpha\in\rrr$, the matching categories $\partial\left(\alpha/\rrr^{\ot}\right)$ are empty and connected. Note that if $\rrr$ has a terminal object $\star$ such that for all $\alpha\in\rrr$, the unique map $\alpha \to \star$ belongs to $\rrr^{\ot}$, then $\rrr$ is left fibrant.
\end{Rem}

\par The conditions of the above definition are purely combinatorial and very tractable. As we are going to make use of it later, let us make the latching category involved in the definition of right fibration more explicit. For fixed $\alpha\in \rrr$ and $f:\beta \to F(\alpha)$, its objects are triples $(\alpha'\in R, f':\beta \to F(\alpha'), g: \alpha' \hookrightarrow \alpha)$ with $g\neq \id$ in $\mathcal{R}^{\to}$ such that $f=F(g)\circ f'$ which we represent by the diagram
\[
\begin{tikzcd}[row sep=1ex]
& F(\alpha')\arrow[hook, start anchor=east]{rd}{}\arrow[hook,near start, start anchor=east]{rd}{F(g)}\\
\beta \arrow[end anchor=west]{ru}{f'}\arrow[swap]{rr}{f}& & F(\alpha)\\
\end{tikzcd}
\]
\noindent and the morphisms $(\alpha', f', g')\to (\alpha'', f'', g'')$ are morphisms $h:\alpha' \to \alpha''$ in $R^{\to}$ such that $F(h)\circ f'=f''$ and $g''\circ h=g'$.
\par Recall that a functor $F: \mathcal{C}\rightarrow \mathcal{D}$ between categories is called a \emph{discrete fibration} if for all $C\in \mathcal{C}$ and all $f: D\rightarrow F(C)$ in $\mathcal{D}$, there is a unique $g: \bar{C}\rightarrow C$ in $\mathcal{C}$ such that $F(g) = f$.

\begin{Prop}
\label{Prop:DirectDivImpliesRightFib}
Let $F: \rrr\rightarrow \sss$ be a morphisms of Reedy categories. If the restriction $F^{\to}: \rrr^{\to}\rightarrow \sss^{\to}$ is a discrete fibration of categories, then $F: \rrr\rightarrow \sss$ is a right fibration of Reedy categories.
\end{Prop}
\begin{proof}
Let us fix $\alpha\in\rrr$ and $f: \beta\rightarrow F(\alpha)$ in $\mathcal{S}$. Since $F^{\to}$ is a discrete fibration and $\sss$ is a Reedy category, we have a unique factorization of $f$:
\begin{equation}\label{equation:RightFibrationProp1}
\begin{tikzcd}[row sep=1ex]
& F(\Bar{\alpha}) \arrow[hook, start anchor=east]{dr}{}\arrow[hook,near start, start anchor=east]{dr}{F(g^{\to})}\\
\beta \arrow[end anchor=west, two heads]{ru}{f^{\ot}}\arrow[swap, two heads, end anchor=west]{rr}{f}& & F(\alpha)\\
\end{tikzcd}
\end{equation}
with $f^{\ot}$ in $\sss^{\ot}$ and $g^{\to}: \bar{\alpha}\rightarrow \alpha$ in $\rrr^{\to}$. We show that this is the initial object of the undercategory $\left((\beta/F)^{\to}/f\right)$. Consider an arbitrary object (represented in a slightly different manner):
\begin{equation}\label{equation:RightFibrationProp2}
\begin{tikzcd}[row sep=1ex]
\beta \arrow[]{rr}{f}\arrow[swap, end anchor=west]{rd}{f'}& & F(\alpha)\\
& F(\alpha') \arrow[hook, start anchor=east]{ru}{}\arrow[swap, hook,near start, start anchor=east]{ru}{F(g')}  & \\
\end{tikzcd}
\end{equation}
\noindent We take the $(\sss^{\ot},\sss^{\to})$-factorization of $f'$:
\[
\begin{tikzcd}[row sep=1ex]
& \Bar{\alpha}'\\
\beta \arrow[two heads, end anchor=west]{ru}{f'^{\ot}}\arrow[]{rr}{f}\arrow[swap, end anchor=west]{rd}{f'}& & F(\alpha)\\
& F(\alpha) \arrow[hook, start anchor=east]{ru}{}\arrow[swap, hook,near start, start anchor=east]{ru}{F(g')}\arrow[hook, crossing over, from=uu, near start]{}{f'^{\to}}  & \\
\end{tikzcd}
\]
\noindent but now by uniqueness of factorization we know $f'^{\ot}=f^{\ot}$, $\Bar{\alpha}'=F(\Bar{\alpha})$ and $F(g')\circ f'^{\to} = F(g^{\to})$. Again since $F^{\to}$ is a discrete fibration, we can write $f'^{\to} = g'^{\to}$ for some $g^{\to}: \bar{\alpha}\rightarrow \alpha'$ in $\rrr^{\to}$ with $g'\circ g'^{\to} = f^{\to}$. Thus we obtain the uniquely defined map
\[
\begin{tikzcd}[row sep=1ex]
& F(\Bar{\alpha})\arrow[hook]{rd}{F(f^{\to})}\\
\beta \arrow[two heads, end anchor=west]{ru}{f'^{\ot}}\arrow[]{rr}{f}\arrow[swap, end anchor=west]{rd}{f'}& & F(\alpha)\\
& F(\alpha') \arrow[hook, start anchor=east]{ru}{}\arrow[swap, hook,near start, start anchor=east]{ru}{F(g)}\arrow[hook, crossing over, from=uu, near start]{}{F(f'^{\to})}  & \\
\end{tikzcd}
\]

Finally, if $f^{\to}\neq \id$, the initial object \eqref{equation:RightFibrationProp1} belongs to the latching category $\partial\left((\beta/F)^{\to}/f\right)$, whereby it is connected. If $f^{\to} = \id$, then it follows for an arbitrary object \eqref{equation:RightFibrationProp2} that $F(g) = \id$ as well. Again since $F$is a discrete fibration, this implies that $g = \id$. Thus \eqref{equation:RightFibrationProp2} does not belong to $\partial\left((\beta/F)^{\to}/f\right)$, whereby it is empty.
\end{proof}

\par Let us end this subsection by establishing some terminology on how monoidal and Reedy structures interact, inspired by \cite{bacard2013colax}.
\begin{Def}
\label{Def:CompatibleDirectDivisible}
Let $\rrr$ be a Reedy category with a monoidal structure $(\vee, I)$.
\begin{enumerate}
    \item The monoidal structure $(\vee,I)$ is \emph{compatible} if the functor $\vee$ is a morphism of Reedy categories $\rrr \times\rrr\to\rrr$.
    \item The category $\rrr$ is \emph{direct divisible with respect to $\vee$} if $\vee^{\to}$ is a discrete fibration, i.e., for any map $f:X \to Y_1\vee Y_2\in \rrr^{\to}$, there exist a unique pair of maps in $\rrr^{\to}$ $f_1: X_1 \to Y_1$ and $f_2: X_2 \to Y_2$ such that $f=f_1\vee f_2$.
    \item $(\rrr,\vee,I)$ is \emph{simple} if for any $\alpha,\beta\in\rrr$, $\deg \left(\alpha\vee \beta\right)=\deg\alpha+\deg\beta$.
\end{enumerate} 
\end{Def}
\begin{Rem}
\label{Rem:CompareBacard}
What we call a Reedy category with compatible monoidal structure is exactly a one-object \emph{locally Reedy 2-category} from \cite{bacard2012lax}, and similarly for a simple Reedy category. Bacard defines direct divisibility only for simple Reedy categories, by requiring $\vee^{\to}$ to be a Grothendieck fibration. However in the simple case, this is equivalent to $\vee^{\to}$ being a discrete fibration.
\end{Rem}

\subsection{Reedy monoidal model structures}

\par In this subsection we highlight the usefulness and inspiration of the definitions of the previous subsection in the theory of model categories. Standard references for model category theory are \cite{hovey1999model} and \cite{hirschhorn2003model}.

\begin{Def}[\cite{hovey1999model}, Def 5.2.2]
\label{Def:LatchingMatchingObj}
Let $\mathcal{R}$ be a Reedy category, $\textbf{A}$ a bicomplete category, $\alpha\in \mathcal{R}$ and $X\in\Fun\left(\mathcal{R}, \textbf{A}\right)$. The \emph{latching object} of $X$ at $\alpha$ is
\[
L_\alpha X=\colim_{\beta\to \alpha\in \partial\left(\mathcal{R}^\to /\alpha\right)} X(\beta)
\]
\noindent and the \emph{matching object} of $X$ at $\alpha$ is
\[
M_\alpha X=\lim_{\alpha\to \beta\in \partial\left(\alpha / \mathcal{R}^\ot \right)} X(\beta)\]
\end{Def}

\par Firstly we present the very well known fact that for any model category $\textbf{A}$, $\Fun\left(\rrr^{\op}, \textbf{A}\right)$ can be endowed with a model structure. Note that the latching and matching objects and categories are now referring to $\rrr^{\op}$ (see Example \ref{Ex:ReedyCat}.3).

\begin{Thm}[\cite{hovey1999model}, Theorem 5.2.5]
\label{Thm:ReedyModelStructure}
Let $\mathcal{R}$ be a Reedy category and $\textbf{A}$ a model category. Then there exists a model structure on $\Fun\left(\mathcal{R}^{\op}, \textbf{A}\right)$, denoted $\Fun\left(\mathcal{R}^{\op}, \textbf{A}\right)_{\Reedy}$, such that
\begin{enumerate}
    \item a map $X \xrightarrow{f} Y$ is a weak equivalence if and only if $X_\alpha \xrightarrow{f_\alpha} Y_\alpha$ is a weak equivalence for all $\alpha\in\mathcal{R}$.
    \item a map $X \xrightarrow{f} Y$ is a (trivial) cofibration if and only if $X_\alpha \coprod_{L_\alpha X}L_\alpha Y \to Y_\alpha$ is a (trivial) cofibration for all $\alpha\in\mathcal{R}$.
    \item a map $X \xrightarrow{f} Y$ is a (trivial) fibration if and only if $X_\alpha \to M_\alpha X \times_{M_\alpha Y} Y_\alpha$ is a (trivial) fibration for all $\alpha\in\mathcal{R}$.
\end{enumerate}
\end{Thm}

\begin{Prop}[\cite{Barwick}, Definition 3.16.3]
\label{Prop:RightFib}
A morphism of Reedy categories $\mathcal{R}\xrightarrow{F}\mathcal{S}$ is a right fibration if and only if for any model category $\textbf{A}$ the adjunction
\[
\begin{tikzcd}
\Fun\left(\mathcal{S}^{\op}, \textbf{A}\right)_{\Reedy} \arrow[swap, yshift=-0.5ex]{r}{F^*}& \Fun\left(\mathcal{R}^{\op}, \textbf{A}\right)_{\Reedy} \arrow[yshift=0.5ex, swap]{l}{F_!}
\end{tikzcd}
\]
\noindent is a Quillen adjunction. 
\par It is a left fibration if and only if for any model category $\textbf{A}$ the adjunction
\[
\begin{tikzcd}
\Fun\left(\mathcal{R}^{\op}, \textbf{A}\right)_{\Reedy} \arrow[yshift=-0.5ex,swap]{r}{F_*}& \Fun\left(\mathcal{S}^{\op}, \textbf{A}\right)_{\Reedy} \arrow[yshift=0.5ex,swap]{l}{F^*}
\end{tikzcd}
\]
\noindent is a Quillen adjunction. 
\end{Prop}

\begin{Cor}
\label{Cor:CofibrantConstants}
A Reedy category $\rrr$ is left fibrant if and only if $\rrr^{\op}$ has cofibrant constants, i.e., for any cofibrant $X\in \textbf{A}$, the constant functor $\Delta X$ is cofibrant in $\Fun\left(\mathcal{R}^{\op}, \textbf{A}\right)_{\Reedy}$.
\end{Cor}
\begin{proof}
Immediate from Proposition \ref{Prop:RightFib}.
\end{proof}

\par Our particular interest in right fibrations comes from the following theorem.

\begin{Thm}[\cite{Barwick}, Corollary 3.50]
\label{Thm:RightFibImpliesQuillen}
Suppose $\textbf{A}$ is a symmetric monoidal model category and $\rrr$ is a Reedy category equipped with a monoidal product
\[\vee: \rrr\times \rrr \to \rrr\]
\noindent that is a right fibration of Reedy categories. Then the Day convolution product
$\dayconv$ is a Quillen bifunctor with respect to the Reedy model structure.
\end{Thm}

\par Note that this theorem only concerns the monoidal product. However, in a monoidal model category also the monoidal unit is required to be homotopically well-behaved.

\begin{Def}[\cite{hovey1999model}, Definition 4.2.6]
\label{Def:ModelMonoidal}
A \emph{monoidal model category} is a monoidal closed category $(\textbf{A},\otimes,\mathbb{I})$ with a model structure on $\textbf{A}$, such that the following conditions hold.
\begin{enumerate}
    \item The monoidal product $\otimes: \textbf{A}\times \textbf{A} \to \textbf{A}$ is a Quillen bifunctor.
    \item Let $q: Q\mathbb{\mathbb{I}} \overset{\sim}{\twoheadrightarrow} \mathbb{\mathbb{I}}$ be the cofibrant replacement of the unit $\mathbb{\mathbb{I}}$. Then the natural maps $q\otimes X: Q\mathbb{\mathbb{I}}\otimes X \to \mathbb{\mathbb{I}}\otimes X$ and $X\otimes q: X \otimes Q\mathbb{\mathbb{I}} \to X \otimes \mathbb{\mathbb{I}}$ are  weak equivalences for all cofibrant $X$. 
\end{enumerate}
\end{Def}

\begin{Rem}
\label{Rem:AnyResolution}
Observe that by standard model category theoretical arguments, condition (b) is equivalent to the statement: for \emph{any} cofibrant replacement $\Bar{q}: \Bar{\mathbb{I}} \xrightarrow{\sim} \mathbb{I}$ of the monoidal unit $\mathbb{I}$ and any cofibrant object $X$, $X\otimes \overline{q}$ and $\overline{q}\otimes X$ are weak equivalences.
\end{Rem}
\par In the spirit of Proposition \ref{Prop:RightFib} and \cite[Definition 15.10.1]{hirschhorn2003model}, we establish the following definition.

\begin{Def}
\label{Def:ReedyMonoidal}
Let $\rrr$ be a Reedy category with a compatible monoidal structure $\left(\vee, I\right)$. The category $\rrr$ is \emph{Reedy monoidal} if for any symmetric monoidal model category $\textbf{A}$, $\Fun\left(\rrr^{\op}, \textbf{A}\right)_{\Reedy}$ is a monoidal model category when equipped with Day convolution.
\end{Def}

\begin{Thm}
\label{Thm:ReedyMonoidal}
Let $\rrr$ be a left fibrant Reedy category with a compatible monoidal structure $(\vee,I)$. If $I$ is a terminal object and $\rrr$ is direct divisible with respect to $\vee$, then $\rrr$ is Reedy monoidal.
\end{Thm}
\begin{proof}
\par By Proposition \ref{Prop:DirectDivImpliesRightFib} and Theorem \ref{Thm:RightFibImpliesQuillen} we obtain condition (a) of Definition \ref{Def:ModelMonoidal}. Let $(\textbf{A},\otimes,\mathbb{I})$ be a symmetric monoidal model category. Since $I$ is terminal, the Day convolution unit of $\Fun\left(\rrr^{\op}, \textbf{A}\right)$ is the constant diagram $\Delta\mathbb{I}$. Let $q: Q\mathbb{I} \to \mathbb{I}$ be the cofibrant replacement of $\mathbb{I}$ in $\textbf{A}$. Then by Corollary \ref{Cor:CofibrantConstants}, $\Delta\left(Q\mathbb{I}\right)$ is cofibrant and moreover $\Delta\left(Q\mathbb{I}\right) \to \Delta \mathbb{I}$ is a levelwise weak equivalence, and thus a weak equivalence in $\Fun\left(\rrr^{\op}, \textbf{A}\right)_{\Reedy}$. Finally, observe that $\left(X \dayconv \Delta q\right)_\alpha\cong X_{\alpha}\otimes q$ for all $\alpha\in \rrr$ and $X\in \textbf{A}$ and thus the final result follows by the assumption that $\textbf{A}$ is monoidal model and by \cite[Remark 5.1.7]{hovey1999model}.
\end{proof}

\section{The category of necklaces}

\par Let us now turn our attention to the category $\nec$ of necklaces, as considered in \cite{baues1980geometry}\cite{dugger2011rigidification}. In $\S 3.1$, we give a self-contained account of a number of relevant factorization classes of morphisms, making use of a combinatorial description of necklaces put forth in \cite{grady2023extended}\cite{lowen2023enriched}. In $\S 3.2$, we go on to show that the $(\text{epi},\text{mono})$-factorization can be made into a Reedy monoidal structure (see Definition \ref{Def:ReedyMonoidal}) with an appropriate degree function (Theorem \ref{Thm:NecIsReedyMonoidal}). This makes use of the notion of dimension of a necklace, which plays a fundamental role in \cite{rivera2018cubical}. The interaction between the classes of epimorphisms and monomorphisms and the dimension, which constitutes an important part of our proof, is implicit in loc. cit as we explain in Remarks \ref{Rem:x} and \ref{Rem:y}, and a decomposition of $\nec$ by degree is used in \cite{CombinatorialPath}.

\subsection{A combinatorial description of necklaces}

\par Let us denote by $\SSet_{*,*} = (\partial\Delta^{1}/\SSet)$ the category of bipointed simplicial sets. This is a monoidal category when equipped with the wedge product
\[
\left(X, x_0,x_1\right)\vee \left(Y, y_0,y_1\right)=\left(X\coprod_{x_1=y_0} Y, x_0, y_1\right)
\]
\noindent and with the $0$-simplex $\left(\Delta^0,0,0\right)$ as monoidal unit . We consider the standard $n$-simplices $\Delta^{n}$ as bipointed at $0$ and $n$. To simplify notation, and as it will not bring about any confusion, we will simply denote $\left(\Delta^n, 0, n\right)$ by $\Delta^n$.

\begin{Def}
\label{Def:Necklaces}
The \emph{category of necklaces} $\nec$ is the full subcategory of $\SSet_{*,*}$ spanned by objects $X$ of shape
\[
X=\Delta^{n_1}\vee \cdots \vee \Delta^{n_k}
\]
\noindent with $k\geq 0$ and $n_i> 0$, called \emph{necklaces}. The simplices $\Delta^{n_i}$ are the \emph{beads} of the necklace $X$. If $k=0$, the corresponding necklace is denoted $\Delta^0$. We call the number of beads $k$ the \emph{bead length} of $X$, denoted $\ell(X)$; and the sum $n_1+\cdots+n_k$ the \emph{spine length} of $X$, denoted $\|X\|$.
\end{Def}

\par Independently in \cite[Proposition 3.4.2]{grady2023extended} and \cite[Proposition 3.4]{lowen2023enriched}, the following alternative combinatorial description of necklaces was given. Consider the subcategory $\fint\subseteq \simp$ containing all morphisms $f: [m]\rightarrow [n]$ such that $f(0) = 0$ and $f(m) = n$.

\begin{Prop}
\label{Prop:CombinatorialNecklaces}
\par The category of necklaces $\nec$ is equivalent to the category defined as follows:
\begin{itemize}
    \item the objects are pairs $(T,p)$ with $p \geq 0$ and $\left\{0,p\right\} \subseteq T \subseteq [p] $;
    \item the morphisms $(T,p) \to (S,q)$ are morphisms $f:[p]\to [q]\in \fint$ such that $S\subseteq f(T)$;
    \item composition and identities are given as in $\Delta_f$.
\end{itemize}
\noindent Moreover, under this equivalence, the wedge $\vee$ corresponds to 
\[\left(T,p\right)\vee\left(S,q\right)=\left(T \cup (p+S),p+q\right)\]
\noindent where $p + S = \left\{p + s \text{ }\rvert\text{} s\in S \right\}$.
\end{Prop}

\par Under this equivalence of categories, we can identify a ``combinatorial'' necklace $(\{0=t_0<t_1<\cdots<t_k=p\}, p)$ with the ``geometric'' necklace $\Delta^{t_1-t_0}\cvee \Delta^{t_k-t_{k-1}}$. Note that in particular, $(\{0 < n\},n)$ corresponds to the $n$-simplex $\Delta^{n}$, while $([p],p)$ corresponds to a sequence of edges, that we call \emph{spines}. In the rest of this section we establish a correspondence between the combinatorial and geometric descriptions of several natural concepts. 

\begin{Def}[\cite{lowen2023enriched}, Def. 3.5]
\label{Def:ActiveInert}
Let $f:(T, p) \to (S, q)$ be a map of necklaces. The map $f$ is \emph{inert} if $p=q$ and $f=\id_{[p]}$. It is \emph{active} if $f(T)=S$.
\end{Def}

\par Consider an arbitrary necklace map $f:(T, p) \to (S, q)$. Then it factors uniquely as an active map followed by an inert map:

\begin{equation}
\label{eq:ActInertFact}
   (T,p)\xrightarrow{f^{\ac}} (f(T), q) \xrightarrow{f^{\ine}} (S,q)
\end{equation}

\begin{Lem}
A map of necklaces $X=(T,p) \xrightarrow{f} Y=(S,q)$ is inert if and only if it is the wedge $\iota_1\vee\cdots\vee \iota_k$ of the inclusions
\[
\iota_i: \Delta^{n_i^1}\vee\cdots\vee\Delta^{n_i^{l_i}}\hookrightarrow\Delta^{n_i^1+\cdots+{n_i^{l_i}}}
\]
\end{Lem}
\begin{proof}
By hypothesis $p=q$ so we are mainly concerned with the inclusion $S\subseteq T$. Denote $ S=\{0=s_0<s_1<\cdots <s_k=q\}$ and write
\[
T=\{0<\cdots<s_i=t_{i}^0<t_{i}^1<\cdots < t_{i}^{l_i}=s_{i+1}<\cdots<p\}
\]
\noindent Then we can set $n_i^j=t^j_i-t^{j-1}_i$ and thus $n_i^1+\cdots+{n_i^{l_i}}=s_i-s_{i-1}$.
\end{proof}
\begin{Lem}
A map of necklaces $X=(T,p) \xrightarrow{f} Y=(S,q)$ is active if and only if it is the wedge $f_1\vee\cdots\vee f_k$ of necklace maps $f_i: \Delta^{n_i}\to \Delta^{m_i}$ induced by morphisms $[n_{i}]\rightarrow [m_{i}]$ in $\fint$.
\end{Lem}
\begin{proof}
\par Simply write $n_i=t_i-{t_{i-1}}$ and $f_i=f\rvert_{\{t_{i-1}<\cdots<t_i\}}$.
\end{proof}

\par Again consider an arbitrary necklace map $f:(T,p)\to (S,q)$. We can factor ${[p]} \to {[q]}$ as a surjective map followed by an injective map, ${[p]} \xrightarrow{f_1} {[r]} \xrightarrow{f_2} {[q]}$. We thus get a new way of uniquely factoring necklace maps
\begin{equation}
\label{eq:ActSurjInjFact}
    (T,p) \xrightarrow{f_1} (f_1(T),r) \xrightarrow{f_2} (S,q)
\end{equation}
\noindent as $f_1(T)\subseteq f_1(T)$ and $S \subseteq f(T)=f_2(f_1(T))$. We note that the first map is active and surjective on vertices and the second is injective on vertices. 

\begin{Lem}
\label{prop:InjectiveNecklaceMap}
A map of necklaces $X=(T,p) \xrightarrow{f} Y=(S,q)$ is a monomorphism if and only if ${[p]}\to {[q]}$ is injective.
\end{Lem}
\begin{proof}
\par Under the identification provided by Proposition \ref{Prop:CombinatorialNecklaces}, ${[p]}$ and ${[q]}$ are the sets of vertices of $X$ and $Y$ respectively, and so the necessity is immediate. 
\par We now show the sufficiency. Factor $f=f^{\ine}\circ f^{\ac}$. Notice that an inert map is the wedge of inclusions, and thus a monomorphism. Now $f^{\ac}$ is the wedge of simplicial maps, which are completely determined by the values at their vertices. If ${[p]}\to {[q]}$ is injective, so will the restriction to the beads be. Thus $f^{\ac}$ is a wedge of monomorphisms and thus a monomorphism. This shows $f$ is a monomorphism as desired.
\end{proof}
\begin{Lem}
\label{prop:SurjectiveNecklaceMap}
A map of necklaces $X=(T,p) \xrightarrow{f} Y = (S,q)$ is an epimorphism if and only if it is active and ${[p]}\to {[q]}$ is surjective.
\end{Lem}
\begin{proof}
\par Again factor $f$ as $f=f^{\ine}\circ f^{\ac}$ with $f^{\ac}$ active and $f^{\ine}$ inert. As $f$ is an epimorphism, so is $f^{\ine}$, but as inert maps are injective, $f^{\ine}=\id$ and thus $f$ is active. Again, as ${[p]} \to {[q]}$ is the map on vertices, it will necessarily be surjective.
\par Assume now that $f$ is active and that ${[p]}\to {[q]}$ is surjective. We can write $f$ as a wedge of epimorphisms between beads, and those are epimorphic.
\end{proof}

\par Recall that as $\SSet_{\star,\star}$ is a topos, the $(\text{epi},\text{mono})$ orthogonal factorization system through the image exists. Via the two lemmas above we deduce the following proposition.
\begin{Prop}
The factorization given in \eqref{eq:ActSurjInjFact} is the restriction of the $(\text{epi}, \text{mono})$-factorization of $\SSet_{\{\star,\star\}}$ to $\nec$.
\end{Prop}
\par As a corollary, we obtain a different proof the following result from \cite{dugger2011rigidification}.

\begin{Cor}[\cite{dugger2011rigidification}, Lemma 3.5(7)]
\label{cor:ImageOfNecklaceMap}
Let $X\xrightarrow{f} Y$ be a map of necklaces. Then $\im f$ is a necklace.
\end{Cor}
\par Let us collect the four factorization classes we introduced in the following table, with descriptions in terms of both the combinatorial and geometric description of $\nec$.
\begin{table}[H]
\makebox[\textwidth][c]{
\begin{tabular}{ c|c|c } 
  & Combinatorial & Geometric\\
 \hhline{=|=|=}
 $f:X \to Y$  & $f:(T,p) \to (S,q)$ & $f:\Delta^{n_1}\vee\cdots\vee\Delta^{n_k} \to \Delta^{m_1}\vee\cdots\vee\Delta^{m_l}$ \\ \hline
 $f$ is active & $S=f(T)$ & $k=l,f=f_1\vee\cdots\vee f_k$, $f_i:\Delta^{n_i}\to \Delta^{m_i}$ \\
$f$ is inert & $p=q, {[p]} \xrightarrow{\id} {[q]}$ & $f = \iota_{1}\vee \cdots\vee \iota_{l}$, $m_{i} = n_{j_{i-1}+1} +\cdots + n_{j_{i}}$\\
& & $\iota_{i}: \Delta^{n_{j_{i}+1}}\vee\cdots \vee \Delta^{n_{j_{i+1}}}\hookrightarrow \Delta^{m_{i}}$\\
\hline
$f$ is epi & ${[p]}\to{[q]}$ surjective, $S=f(T)$ & $f_n$ is surjective, $\forall n\geq 0$\\
$f$ is mono & ${[p]}\to{[q]}$ injective & $f_n$ is injective, $\forall n\geq 0$\\
 \hline
\end{tabular}}
\caption{}
\label{table:maps}
\end{table}

\par We now introduce a further decomposition of the epimorphisms, that will come in handy in the next section.
\begin{Def}
Let $f = f_{1}\cvee f_k$ be an active surjective necklace map with $f_i: \Delta^{n_{i}}\to \Delta^{m_{i}}$. The map $f$ is \emph{bead reducing} if $m_i\geq 1$ for all $1\leq i \leq k$ and \emph{spine collapsing} if for all $1\leq i \leq k$, $f_i=\id_{\Delta^{n_{i}}}$ or $f_i: \Delta^{1}\rightarrow \Delta^{0}$.
\end{Def}
\par Observe that if $(T,p)\xrightarrow{f}(f(T), q)$ is an active surjective map with $T=\{0=t_0<\cdots< t_k=p\}$, then we have for all $0<i\leq k$
\[
0\leq f(t_i)-f(t_{i-1})\leq t_i-t_{i-1}
\]

\begin{Lem}
A map of necklaces $X=(T,p)\xrightarrow{f}(f(T), q)$ is bead reducing if for all $0< i \leq k$, $0<f(t_{i})-f(t_{i-1})$.
\end{Lem}
\begin{proof}
Follows immediately as $m_i=f(t_i)-f(t_{i-1})$. 
\end{proof}
\begin{Lem}
A map of necklaces $X=(T,p)\xrightarrow{f}(f(T), q)$ is spine collapsing if and only if for all $0< i \leq k$, $f(t_{i})-f(t_{i-1})=t_{i}-t_{i-1}$ or $t_i-t_{i-1}=1$ and $f(t_{i})-f(t_{i-1})=0$.
\end{Lem}
\begin{proof}
The first possibility corresponds to $f_i=\id$ and the second to $f_i: \Delta^{1}\to \Delta^{0}$.
\end{proof}

\par It clearly follows from the previous lemmas that any active surjective necklace map $(T,p) \to (f(T), q)$ can be factored as a bead reducing map followed by a spine collapsing map. We thus obtain a factorization of any necklace map as follows
\begin{equation}
    \bullet \xrightarrow{\text{bead reducing}} \bullet \xrightarrow{\text{spine collapsing}} \bullet \xrightarrow{\text{monomorphism}}\bullet 
\end{equation}

\begin{Rem}
\label{Rem:x}
\par In \cite[Proposition 3.1]{rivera2018cubical}, a finer factorization of necklace maps is described, in terms of maps of types $(i)$, $(ii)$ and $(iii)$. It is readily seen that a map is a monomorphism (resp. bead reducing, resp. spine collapsing) precisely when it is a compositions of maps of type $(i)$ (resp. $(ii)$, resp. $(iii)$).
\end{Rem}

\subsection{The Reedy monoidal structure on necklaces}
\par In this section, in order to show that the $(\text{epi}, \text{mono})$-factorization extends to a Reedy monoidal structure on $\nec$, we make use of the following fundamental notion.

\begin{Def}[\cite{CombinatorialPath} \S 2, \cite{rivera2018cubical} \S 4]
Let $X\in \nec$. Its \emph{dimension} is 
$$
\dim \left(X \right) = \|X\| - \ell(X)
$$
\end{Def}

\begin{Rem}
It is readily seen that for any $X,Y\in\nec$, we have $\| X\vee Y\| = \| X\| + \| Y\|$ and $\ell(X\vee Y) = \ell(X) + \ell(Y)$ and consequently $\dim\left(X\vee Y\right)=\dim X + \dim Y$.
\end{Rem}

\begin{Lem}
\label{Lem:InjectiveAndDegree}
Let $f:X \to Y$ be a non-identity monomorphism. Then $\dim X < \dim Y$.
\end{Lem}
\begin{proof}
Factor $f$ as $f=f^{\ine}\circ f^{\ac}$. As $f\neq \id$ then either $f^{\ac}\neq \id $ or $f^{\ine}\neq \id$. In the first case, ${[p]}\to {[q]}$ is a strict injection, thus $q>p$, and $|T|=|f(T)|$, so $\dim X<\dim Z$. In the second case, we know that $S \subsetneq f(T)$ so $|f(T)|>|S|$ and finally $\dim Z < \dim Y$.
\end{proof}

\begin{Lem}
\label{Lem:BeadRedDegreeAndLenght}
Let $f: X\to Y$ be a non-identity bead reducing map. Then $\dim X < \dim Y$.
\end{Lem}
\begin{proof}
\par Write $f:X=(T,p) \to Y=(f(T), q)$. We begin by noting that as $0<f(t_i)-f(t_{i-1})$ for all $i$, then $|T|=|f(T)|$. As $f\neq \id$, then $p>q$ and the result follows.
\end{proof}
\begin{Lem}
\label{Lem:SpineCollDegreeAndLenght}
Let $f: X\to Y$ be a non-identity spine collapsing necklace map. Then $\dim X =\dim Y$ and $\ell(X)>\ell(Y)$.
\end{Lem}
\begin{proof}
\par The equality follows from the additivity of dimension and the fact that $\dim\Delta^1=0$. As for the length, writing $f=f^1\cvee f^k$, and as $f\neq\id$, there exists at least one $1\leq i \leq k$ such that $f^i:[1]\to [0]$ and thus $\ell(X)>\ell(Y)$.
\end{proof}

\begin{Rem}
\label{Rem:y}
In the light of Remark \ref{Rem:x}, Lemmas \ref{Lem:InjectiveAndDegree}, \ref{Lem:BeadRedDegreeAndLenght} and \ref{Lem:SpineCollDegreeAndLenght} can also be deduced from the proof of \cite[Prop.4.2]{rivera2018cubical}.
\end{Rem}
\par We collect the behaviour of spine length, bead length and dimension with respect to the several classes of maps introduced so far in the next table.
\begin{table}[H]
\centering
\begin{tabular}{c|c|c|c} 
$f: X\xrightarrow{\neq} Y$  & spine length & {bead length} & dimension\\
 \hhline{=|=|=|=}
 bead reducing  & $\| X\| > \| Y\| $ & $\ell(X) = \ell(Y)$ & $\dim(X) > \dim(Y)$ \\ 
 spine collapsing & $\| X\| > \| Y\|$ & $\ell(X) > \ell(Y)$ & $\dim(X) = \dim(Y)$ \\
 \hline
active injective & $\| X\| < \| Y\|$ & $\ell(X) = \ell(Y)$ & $\dim(X) < \dim(Y)$\\
inert & $\| X\| = \| Y\|$ & $\ell(X) > \ell(Y)$ & $\dim(X) < \dim(Y)$\\
 \hline
\end{tabular}
\caption{}
\label{table:relations}
\end{table}

\begin{Def}
The \emph{degree function} is $\deg: \Ob \nec \to \mathbb{N}\times \mathbb{N}$ given by $\deg X = \left(\dim X, \ell( X)\right)$ 
\end{Def}

\begin{Thm}
\label{Thm:NecIsReedy}
The category $\nec$ equipped with $\nec^{\ot}=\{\text{epimorphisms}\}$, $\nec^{\to}=\{\text{monomorphisms}\}$ and $\deg$ is a Reedy category.
\end{Thm}
\begin{proof}
\par We already showed that $(\nec^{\leftarrow}, \nec^{\rightarrow})$ provides unique factorizations. Equip $\mathbb{N}\times \mathbb{N}$ with the lexicographical order, so that it may be identified with the ordinal $\omega^2$. Lemmas \ref{Lem:BeadRedDegreeAndLenght} and \ref{Lem:SpineCollDegreeAndLenght} show that every non-identity morphism in $\nec^{\ot}$ lowers degree and Lemma \ref{Lem:InjectiveAndDegree} shows that every non-identity morphism in $\nec^\to$ raises degree.
\end{proof}
\begin{Rem}
A decomposition of $\nec$ by degree is present in \cite[\S 2]{CombinatorialPath}.
By inspection of Table \ref{table:relations}, one observes that other possible degree functions define Reedy structures with the $(\text{epi},\text{mono})$-factorization system (e.g. $\deg X = (\|X\|, \dim X)$). Note that any choice of degree function compatible with the factorization will give rise to an isomorphic Reedy structure.
\end{Rem}

\begin{Thm}
\label{Thm:NecIsReedyMonoidal}
The category $\nec$ is simple Reedy monoidal.
\end{Thm}
\begin{proof}
We make use of Theorem \ref{Thm:ReedyMonoidal}. Using Remark \ref{Rem:TerminalLeftFib} and noting that any map $X\to\Delta^0$ is an epimorphism, we conclude that $\nec$ is left fibrant. Consider a monomorphism $f: X\hookrightarrow Y_1\vee Y_2$ that we write in combinatorial fashion 
\begin{equation}
    (T,p) \xhookrightarrow{f} (S_1, q_1)\vee(S_2,q_2)=(S_1\cup\left(q_1+S_2\right), q_1+q_2)
\end{equation}
\noindent We know that $q_1\in S_1 \cup (S_2+q_1)\subseteq f(T)$ and as $f:{[p]}\to {[q_1+q_2]}$ is injective, there is a unique $r\in T$ such that $f(r)=q_1$. We can thus write $
T_1=\left(T\cap{[r]}, r\right)$, $T_2=(\left(T\cup \{r,\dots,p\}\right)-r, p-r)$ and $f_i=f\rvert_{X_i}$. We have then unique $f_1$ and $f_2$ such that $f=f_1\vee f_2$. We conclude that $\nec$ is direct divisible with respect to $\vee$. Finally, to show $\nec$ is simple we observe that pointwise sum in $\mathbb{N}\times\mathbb{N}$ corresponds to ordinal sum in $\omega^2$ and that both dimension and bead length are additive with respect to $\vee$.
\end{proof}


\printbibliography

\end{document}